\documentclass{llncs}

\usepackage[T1]{fontenc}

\usepackage{graphicx}

\usepackage{mathtools}
\usepackage{amssymb}
\usepackage{xparse}
\usepackage{booktabs}

\usepackage{pgfplots}
\usetikzlibrary{plotmarks}
\usepgfplotslibrary{colorbrewer}

\usepackage{algorithm}
\usepackage{algpseudocode}
\algrenewcommand\algorithmicrequire{\textbf{Input:}}
\algrenewcommand\algorithmicensure{\textbf{Output:}}

\usepackage{xurl}
\usepackage[unicode,pdfusetitle,pdfa]{hyperref}
\makeatletter
\AddToHook{env/algorithmic/before}{\def\@currentcounter{ALG@line}}
\makeatother
\usepackage[capitalize,noabbrev]{cleveref}
\crefalias{ALG@line}{line}

\newcommand*{\citeresult}[2]{\cite[#2]{#1}}

\usepackage{color}

\hypersetup{
	colorlinks   = true, %
	urlcolor     = blue, %
	linkcolor    = blue, %
	citecolor   = blue, %
	bookmarksnumbered,
	linktoc = all,
	bookmarksopenlevel = 1,
}

\newcommand*{\defemph}[1]{\emph{#1}}

\DeclareMathOperator{\Places}{\mathbb{P}}

\NewDocumentCommand{\FinitePart}{r()}{#1^0}
\NewDocumentCommand{\InfinitePart}{r()}{#1^\infty}

\DeclareMathOperator{\Div}{Div}
\DeclareMathOperator{\DivZero}{Div^0}

\DeclareMathOperator{\DivisorClassGroup}{Cl}
\DeclareMathOperator{\Jacobian}{Cl^0}

\DeclareMathOperator{\RiemannRoch}{\mathcal{L}}
\DeclareMathOperator{\RiemannRochDim}{\ell}

\NewDocumentCommand{\DivHeight}{r()}{\operatorname{h}\brac*{#1}}

\DeclareMathOperator{\SSRR}{SSRR}

\NewDocumentCommand{\SSRRComplexity}{r()}{\operatorname{\mathbf{RR}}\brac*{#1}}
\newcommand*{\PartialAdditionComplexity}{\mathbf{I}}

\DeclareMathOperator{\supp}{supp}

\newcommand*{\MaximalOrder}{\mathfrak{O}}

\newcommand*{\FiniteMaximalOrder}{\MaximalOrder_{F, 0}}
\newcommand*{\InfiniteMaximalOrder}{\MaximalOrder_{F, \infty}}

\DeclareMathOperator{\divisor}{div}

\newcommand*{\principaldivisor}[1]{\divisor(#1)}

\DeclarePairedDelimiter{\brac}{(}{)}

\DeclarePairedDelimiter{\ceil}{\lceil}{\rceil}

\DeclarePairedDelimiter{\abs}{|}{|}
\DeclarePairedDelimiter{\set}{\{}{\}}

\newcommand*{\union}{\cup}

\newcommand*{\setst}{:}

\NewDocumentCommand{\bigO}{d()}{\IfNoValueTF{#1}{O}{O\brac*{#1}}}

\newcommand*{\FF}{\mathbb{F}}
\newcommand*{\ZZ}{\mathbb{Z}}
\newcommand*{\QQ}{\mathbb{Q}}

\pgfplotsset{
	compat=1.18,
	colormap/Paired,
	xtick pos=bottom,
	ytick pos=left,
	xtick align=outside,
	legend style={at={(0.5, -0.2)},anchor=north,/tikz/every even column/.append style={column sep=0.5cm}, cells={align=left}, font=\scriptsize},
	legend cell align=left,
	grid=major,
	yminorgrids=true,
	tick align=center,
	log basis y=10,
	major tick style={thick},
	/tikz/every mark/.append style={scale=1},
	every axis plot/.append style={line width=1pt},
	every axis plot post/.append style={
		every mark/.append style={line width=1pt}
	},
	ylabel style={align=center},
	scaled ticks=false,
	samples=101,
}

\pgfplotstableset{
	col sep=comma,
	search path={Data},
}

\newcommand{\EqPartialAdditionComplexity}{n^6 (g \log(gn) \log\log(gn) + g^2)}
\newcommand{\EqSSRRComplexity}{n^5 (h + n^2 C_f)^2}
\newcommand{\EqUniqueHessReductionLinComplexityWorst}{C_f^2 g n^9 + C_f g^2 n^7 + g^3 n^6}
\newcommand{\EqUniqueHessReductionLinComplexityCachingWorst}{C_f^2 g n^9 + C_f g^2 n^7 + g^3 n^5}
\newcommand{\EqUniqueHessReductionLinComplexity}{C_f^2 n^9 + C_f g n^7 + g^2 n^6 + g n^6 \log(gn) \log\log(gn)}
\newcommand{\EqUniqueHessReductionLinComplexityCaching}{C_f^2 n^9 + C_f g n^7 + g^2 n^5}
\newcommand{\EqUniqueHessAdditionLinComplexityWorst}{\EqUniqueHessReductionLinComplexityWorst}
\newcommand{\EqUniqueHessAdditionLinComplexityCachingWorst}{\EqUniqueHessReductionLinComplexityCachingWorst}
\newcommand{\EqUniqueHessAdditionLinComplexity}{\EqUniqueHessReductionLinComplexity}
\newcommand{\EqUniqueHessAdditionLinComplexityCaching}{C_f^2 n^9 + C_f g n^7 + g^2 n^6 + g n^6 \log(gn) \log \log(gn)}
\newcommand{\EqUniqueHessReductionBSComplexity}{\log(g)(\EqUniqueHessAdditionLinComplexityCaching)}

\newcommand{\EqUniqueHessAdditionBSComplexity}{\EqUniqueHessReductionBSComplexity}
\newcommand{\EqUniqueHessAdditionBSComplexityCaching}{\log(g) (\EqUniqueHessReductionLinComplexityCaching) + g^2 n^6 + g n^6 \log(gn) \log \log(gn)}

\begin{document}
\title{Improvements to Jacobian Arithmetic in Global Function Fields}

\author{Vincent Macri\inst{1}\orcidID{0009-0000-5233-7658} \and
Michael Jacobson, Jr.\inst{1}\orcidID{0000-0002-4906-0544} \and
Renate Scheidler\inst{1}\orcidID{0000-0001-7164-8769}}
\authorrunning{V. Macri et al.}

\institute{University of Calgary, Calgary, Canada
\email{\{vincent.macri,jacobs,rscheidl\}@ucalgary.ca}}

\maketitle              %
\begin{abstract}
    We present two improvements to arithmetic in the Jacobian of global function fields based on the approach of Hess.
    The first reduces the number of expensive reduction steps by optimizing for typical inputs rather than worst-case behavior, assuming the function field contains a degree-one place. The second introduces a memory–time trade-off that speeds up computations by caching frequently used intermediate results.
    Our asymptotic analysis and empirical experiments show that our improved algorithms are significantly faster in practice than previously published methods. To the best of our knowledge, our publicly-available software implementation of Jacobian arithmetic is the first to support unique representatives of divisor classes.

\end{abstract}

\section{Introduction}

\subsection{Motivation}
Elliptic curves are the subject of numerous notoriously challenging open problems, including the Koblitz--Zywina conjecture, the Lang--Trotter conjecture, the Sato--Tate conjecture, and perhaps most famously, the Birch and Swinnerton-Dyer conjecture \cite{sutherlandFastJacobianArithmetic2019,wilesBirchSwinnertonDyerConjecture}.
While these conjectures are about elliptic curves over $\QQ$, generating numerical evidence for (or against) them requires considering a given curve over finite fields $\FF_q$ for many values of $q$ and computing the number of points on the curve.
Efficient point arithmetic is essential in this endeavour.
This motivates the desire to compute the size of the group of points on an elliptic curve over a finite field.
Although proofs of these conjectures are out of reach, there is much computational evidence supporting them.

A natural next step is to consider generalizations of these conjectures to curves that are not necessarily elliptic, and see whether computational evidence supports or refutes the generalized conjectures.
In the more general setting, the points on the curve no longer form a group: instead, elements of the Jacobian are equivalence classes of degree zero divisors.
Hence, rather than counting points over many finite fields, one instead needs to compute the order of the Jacobian of the curve over all these fields.
In order to tackle this and similar problems, it is essential to have fast algorithms for computing in the Jacobian group of a curve or its associated function field.
Moreover, for algorithmic efficiency, it is beneficial if this arithmetic should be performed in a way that gives unique representatives of elements of the Jacobian group.

The most recent efficient algorithm for computing in the Jacobian with unique representatives of elements is the approach of maximally reduced divisors from \cite{hessComputingRiemannRochSpaces2002}.
Other recent approaches to Jacobian arithmetic such as \cite{khuri-makdisiLinearAlgebraAlgorithms2004,khuri-makdisiAsymptoticallyFastGroup2007,jungeAsymptoticallyFastArithmetic2022} do not give unique representatives of Jacobian elements.

\subsection{Our Contributions}
In \cite{hessComputingRiemannRochSpaces2002}, Hess showed how to compute in the divisor class group, which contains the Jacobian $\Jacobian(F)$ as a subgroup.
By specializing to the Jacobian and requiring the existence of a degree one place (a restriction which is nearly always satisfied), we develop an improved algorithm for reducing divisors to a unique reduced representative of their divisor classes that is significantly faster than the approach recommended in \cite{hessComputingRiemannRochSpaces2002}.
We also demonstrate how effective use of caching can speed up Jacobian arithmetic even further, especially if the function field contains a degree one infinite place.
Our improved Jacobian algorithm is both faster asymptotically for typical inputs, and faster in practice when implemented.
Complexity analysis indicates that our algorithm should give a speedup of approximately $\frac{\log_2(g)}{2}$, and a performance comparison confirms this: for example, for degree 3 function fields of genus 100, our implementation of our improved algorithm is approximately three times faster than our implementation of the unique Jacobian arithmetic algorithm suggested in \cite{hessComputingRiemannRochSpaces2002}.
An implementation of our algorithm is available in SageMath \cite{sagemath}, since the 10.9.beta7 release.

\subsection{Outline of the Paper}
In \cref{sec:background} we establish our notation, summarize the necessary background material, and state a few results that will be useful later on.
\Cref{sec:unique_hess} describes our modifications to Hess' maximally reduced divisors from \cite{hessComputingRiemannRochSpaces2002} that facilitate our improved algorithms for Jacobian arithmetic presented in \cref{sec:algorithms}.
We conduct a theoretical complexity analysis in \cref{sec:complexity} before analyzing the actual performance of our implementations in \cref{sec:timing_experiments}.

\section{Background}
\label{sec:background}

\subsection{Notation}
For details, the reader is referred to \cite{stichtenothAlgebraicFunctionFields2009}. 
Throughout, let $K$ be a finite field and $F/K$ a geometric global function field of genus $g$. Then $F = K(x, y)$ with $x, y \in F$, where $K(x)$ is a rational function field and $y$ has minimal polynomial $f(t) = t^n + \sum_{i=0}^{n - 1} t^i a_i(x) \in K[x,t]$, with $a_i(x) \in K[x]$ for $0 \le i \le n-1$. Hence, $F/K(x)$  has degree $n = \deg_t(f)$ and $f(y) = 0$. The quantity 
\begin{equation}
	\label{eq:Cf}
	C_f \coloneqq \max\set*{\ceil*{\frac{\deg(a_i)}{n - i}} \setst 0 \leq i \leq n - 1}
\end{equation}
determines the size of $F$ and will appear in our complexity statements later on. 

We denote the set of places of $F$ by $\Places(F)$, partitioned into the  infinite places (the poles of $x$) and the finite places of $F$. 
The finite and infinite maximal orders of $F/K(x)$ are denoted by $\FiniteMaximalOrder$ and $\InfiniteMaximalOrder$ respectively.

Divisors of $F$ are formal finite sums $D = \sum_{P \in \Places(F)} v_P(D) P$, where $v_P(D) \in \mathbb{Z}$ is the valuation of $D$ at $P$. We write $\supp(D) = \{ P \in \Places(F) : v_P(D) \ne 0\}$ (the support of $D$) and let $\deg(D)$ denote the degree of $D$. The respective sub-sums of $D$ over the finite places and the infinite places in $\supp(D)$ are denoted by $\FinitePart(D)$ and $\InfinitePart(D)$, so $D = \FinitePart(D) + \InfinitePart(D)$. 

The principal divisor of a non-zero function $a \in F$ is denoted $\divisor(a)$. Two divisors $D, D' \in \Div(F)$ are (linearly) equivalent, denoted $D \equiv D'$, if they differ by a principal divisor. Let $\Div(F)$ and $\DivZero(F)$ denote the groups of divisors and degree zero divisors of $F$, respectively, and let $\DivisorClassGroup(F)$ and $\Jacobian(F)$ denote the groups of divisor classes and degree zero divisor classes of $F$, respectively, under linear equivalence. The group $\Jacobian(F)$ is also referred to as the \emph{Jacobian} of $F$ and is the main protagonist of our work herein.

The Riemann-Roch space of a divisor $D \in \Div(F)$ is the finite-dimensional $K$-vector space $\RiemannRoch(D) = \set{a \in F : D + \divisor(a) \geq 0} \union \set{0}$; its $K$-dimension is denoted by $\RiemannRochDim(D)$.
A partial order on $\Div(F)$ is given by $D \leq D'$ if $v_P(D) \leq v_P(D')$ for all $P \in \Places(F)$.
Riemann-Roch spaces and their dimensions are monotonic with respect to this partial order; this monotonicity will play a key role in our main algorithm (\cref{alg:hess_reduction}). 
\begin{lemma}
	\label{lem:riemann_roch_monotonicity}
	\citeresult{stichtenothAlgebraicFunctionFields2009}{Lemma 1.4.8}
	Let $D, D' \in \Div(F)$ with $D \leq D'$.
	Then $\RiemannRoch(D) \subseteq \RiemannRoch(D')$ and 
    $\ell(D') - \ell(D) = \dim(\RiemannRoch(D')/\!\RiemannRoch(D)) \leq \deg(D') - \deg(D)$.
\end{lemma}
We conclude the background section with the notion of a maximally reduced divisor as presented in \cite{hessComputingRiemannRochSpaces2002}. For convenience, we restate the definition here.
\begin{definition} \label{def:HessReduction}
	\citeresult{hessComputingRiemannRochSpaces2002}{Definition 8.1}
	Let $A \in \Div(F)$ with $\deg(A) \geq 1$.
	A divisor $\tilde{D}$ is called \defemph{maximally reduced along $A$} if $\tilde{D} \geq 0$ and $\RiemannRochDim(\tilde{D} - sA) = 0$ for all integers $s \geq 1$.
	The representation of a divisor $D$ as $D = \tilde{D} - rA - \divisor(a)$, with $\tilde{D}$ maximally reduced along $A$, $r \in \ZZ$, and $a \in F^*$ is called a \defemph{maximal reduction of $D$ along $A$}.
\end{definition}
Hess \cite[Proposition 8.2]{hessComputingRiemannRochSpaces2002} showed that if $\deg(A) = 1$ then the maximal reduction of a divisor $D$ along $A$ is unique.

\section{Unique Hess representation} \label{sec:unique_hess}

The notion of maximally reduced divisors (Definition \ref{def:HessReduction}) was originally defined for elements of $\DivisorClassGroup(F)$, but specializing to $\Jacobian(F)$ and imposing the restriction that $A$ is a degree one place achieves improved algorithmic performance.

\begin{definition}
	\label{def:hess_reduced}
	Let $D \in \DivZero(F)$ and $A \in \Places(F)$ with $\deg(A) = 1$.
	Then the \defemph{Hess-reduction of $D$ (along $A$)} is the maximal reduction of $D$ along $A$. 
	We call $D$ \defemph{Hess-reduced (along $A$)} if $D$ is equal to its Hess-reduction (along $A$) and refer to it as the \defemph{Unique Hess representation} of its divisor class.
\end{definition}

We note that the terminology of \cref{def:hess_reduced} is restricted to degree zero divisors. 
By \cite[Proposition 8.2]{hessComputingRiemannRochSpaces2002}, the Hess-reduction of a degree zero divisor is unique. Here, $A$ can be finite or infinite, but the latter case will allow for more efficient implementations later on as it creates additional opportunities to exploit caching of intermediate results.
In practice, requiring the existence of a degree one place is a very mild restriction.
By the Hasse-Weil Bound \cite[Theorem 5.2.3]{stichtenothAlgebraicFunctionFields2009}, every function field $F / \FF_q$ of genus $g$ has a degree one place if $q$ is sufficiently large relative to $g$. Hence, the existence of a degree one place can be guaranteed by extending scalars to a sufficiently large extension of $\FF_q$.
Even for prime fields $\FF_q$, the probability that a function field has no degree one place is quite low; see~\cite{howePointlessCurvesGenus2005} for a heuristic argument.

We require two key results for Hess-reduced divisors.
\begin{proposition}
	\label{prop:Hess_Dtilde_A_free}
	Let $A \in \Places(F)$ with $\deg(A) = 1$, and let $\tilde{D} \in \DivZero(F)$ be Hess-reduced along~$A$.
	Then $A \notin \supp(\tilde{D})$.
\end{proposition}
\begin{proof}
    Since $\tilde{D}$ is maximally reduced along $A$, we have  $\RiemannRochDim(\tilde{D} - A) = 0$. It follows that $\tilde{D} - A \not\geq 0$, so $v_A(\tilde{D}) \leq 0$. On the other hand $\tilde{D} \geq 0$, so $v_A(\tilde{D}) = 0$.
    \qed
\end{proof}
In \cite{hessComputingRiemannRochSpaces2002}, Hess noted that for a divisor $\tilde{D}$ maximally reduced along $A$, we have $\RiemannRochDim(\tilde{D}) \leq \deg(A)$ and $\deg(\tilde{D}) < g + \deg(A)$.
We strengthen these statements for the setting of Hess-reduced divisors.
\begin{proposition}
	\label{prop:Hess_Dtilde_r_bounds}
	Let $A \in \Places(F)$ with $\deg(A) = 1$, and let $D = \tilde{D} - rA$ be Hess-reduced along $A$.
	Then $0 \leq r = \deg(\tilde{D}) \leq g$.
\end{proposition}
\begin{proof}
	Since $\tilde{D}$ is maximally reduced along $A$, we have $\tilde{D} \ge 0$ and $\RiemannRochDim(\tilde{D} - A) = 0$, so \cref{lem:riemann_roch_monotonicity} yields
	\[\RiemannRochDim(\tilde{D}) = \RiemannRochDim(\tilde{D}) - \RiemannRochDim(\tilde{D} - A) \leq \deg(\tilde{D}) - \deg(\tilde{D} - A) = \deg(A) = 1. \]
    Riemann's Theorem \cite[Theorem 1.4.17]{stichtenothAlgebraicFunctionFields2009} now implies
    \[ \deg(\tilde{D}) \le \RiemannRochDim(\tilde{D}) + g - 1 \leq g.  \]
    Finally, since $\deg(D) = 0$ by definition, we have $\deg(\tilde{D}) = \deg(rA) = r$. 
    \qed
\end{proof}

\section{Improvements to the Hess Reduction Algorithm}
\label{sec:algorithms}

In this section, we introduce two new modifications to the algorithm of \cite{hessComputingRiemannRochSpaces2002} for computing maximal reductions that effect substantial efficiency improvements when restricting to degree zero divisors. The first employs a search strategy for finding maximal reductions that is different from Hess's and performs better in practice. The second improvement introduces caching of intermediate computational results which effects a further speed-up.

\subsection{Improvements to Reduction Algorithm}
In \cite{hessComputingRiemannRochSpaces2002}, Hess provided an efficient algorithm for computing a basis of a Riemann-Roch space and described how to use it for finding maximal reductions of divisors. After specializing to degree zero divisors and the framework of Unique Hess representations in the Jacobian of $F$, this task translates to the following problem (see \citeresult{macriMScThesis2025}{Section 4.1} for details).
\begin{problem}[\texttt{HR-Min}]
	Let $A \in \Places(F)$ with $\deg(A) = 1$.
	Given $D \in \DivZero(F)$, find the minimal integer $0 \leq r \leq g$ such that $\RiemannRochDim(D + rA) = 1$.
\end{problem}
By the monotonicity of $\RiemannRochDim$, an equivalent problem is to find the maximal integer $r'$ such that $\RiemannRochDim(D + r'A) = 0$.
Then $r = r' + 1$ solves \texttt{HR-Min}.

Once a solution to \texttt{HR-Min} for an input divisor $D \in \DivZero(F)$ is found, it is fairly straightforward to compute the Hess-reduction of $D$, assuming we can compute Riemann-Roch bases. 
Suppose that $r$ solves \texttt{HR-Min} for $D$ and that we can  compute a non-zero element $a \in \RiemannRoch(D + rA)$, which is then a basis of this space.
Then $\principaldivisor{a}$ is unique because $\RiemannRochDim(D + rA) = 1$.
Moreover, $D + rA + \principaldivisor{a} \geq 0$, so setting $\tilde{D} \coloneqq D + rA + \principaldivisor{a}$, we see that $\tilde{D} \geq 0$ and $\RiemannRochDim(\tilde{D} - sA) = 0$ for all integers $s \geq 1$. So $\tilde{D}$ is maximally reduced along $A$, and hence $D \equiv \tilde{D} - rA$ is the unique Hess-reduction of $D$ along $A$.
See \cref{alg:jacobian_arithmetic} for an algorithmic description of this process.

Finding the solution $r$ to \texttt{HR-Min} can be accomplished by trying each possible candidate value for $r$ between $0$ and $g$, but each trial requires computing a basis of a Riemann-Roch space, which is by far the most expensive part of the computation.
In \cite{hessComputingRiemannRochSpaces2002}, Hess suggested to find $r$ via binary search, which takes $\bigO(\log g)$ iterations, but Hess was working in $\DivisorClassGroup(F)$ rather than $\Jacobian(F)$.
When working over $\Jacobian(F)$, a heuristic argument \cite[Subsection 4.2.1]{macriMScThesis2025} shows that with high likelihood, for a randomly chosen $D \in \DivZero(F)$, the solution to \texttt{HR-Min} is $r = g$.
Empirically, we find that for $K = \FF_q$, the value $r = g$ solves \texttt{HR-Min} with probability approaching $1 - q^{-1}$ as $q$ grows.
This stems from the observation that a random divisor of degree at most $g$ is expected to have degree exactly~$g$ with the same likelihood that a random polynomial of degree at most $g$ in $\FF_q[x]$ is expected to have degree exacly $g$, which is just the probability that its coefficient at $x^g$ is non-zero. 
Furthermore, the same heuristic suggests that when $r = g$ is not the solution to \texttt{HR-Min}, the next most likely candidate is $r = g - 1$, then $g - 2$, and so on, with the solution $r = 0$ corresponding to the case when $D$ is principal.

This observation suggests that a linear search in decreasing order to solving \texttt{HR-Min} should outperform the binary search approach  suggested in \cite{hessComputingRiemannRochSpaces2002}, by significantly decreasing the number of Riemann-Roch basis computations. Our implementation bears this out, and this idea forms the basis of   \cref{alg:hess_reduction}, our improved algorithm for solving \texttt{HR-Min}. The key idea of this algorithm is to try the most likely values for $r$ first. A formal proof of correctness is given in \cref{thm:Correctness-of-reduction}.

The approach for computing the Hess-reduction of a divisor outlined above relies on the ability to compute the Riemann-Roch dimension of a divisor $E$. In practice, this is done by computing a basis of $\RiemannRoch(E)$  and returning the number of basis elements. In our setting, by the monotonicity of the Riemann-Roch dimension (\cref{lem:riemann_roch_monotonicity}), it suffices to only compute one non-zero element of $\RiemannRoch(E)$ rather than a full basis, which may be faster.
An easy modification of the Riemann-Roch basis algorithm from \cite{hessComputingRiemannRochSpaces2002} effects this. We call this variant a \defemph{short-circuited Riemann-Roch} operation, denoted by $\SSRR$. This operation takes as input a divisor $E \in \Div(F)$ and outputs a non-zero element of $\RiemannRoch(E)$ if $\RiemannRochDim(E) > 0$, or the string \texttt{None} if $\RiemannRoch(E) = \{0\}$.
Note that 
if $\RiemannRochDim(E) = 1$, then $\SSRR(E)$ returns a basis of $\RiemannRoch(E)$, and $\principaldivisor{\SSRR(E)}$ is unique

\begin{algorithm}
	\caption{Optimized Hess reduction strategy via linear search}
	\label{alg:hess_reduction}
	\begin{algorithmic}[1]
		\Require $D \in \DivZero(F)$, $A \in \Places(F)$ with $\deg(A) = 1$
		\Ensure $(r, a)$ where $r$ solves \texttt{HR-Min} and $0 \neq a \in \RiemannRoch(D + rA)$

		\State $\tilde{D} \gets D + (g - 1) A$ \label{alg:hess_reduction:add1}
		\State $a \gets \Call{SSRR}{\tilde{D}}$ \label{alg:hess_reduction:ssrr1}
		\If{$a = \texttt{None}$} \Comment{If $\RiemannRochDim(D + (g - 1) A) = 0$ then $\RiemannRochDim(D + gA) = 1$.}
			\State $\tilde{D} \gets \tilde{D} + A$ \Comment{$\tilde{D} = D + gA$} \label{alg:hess_reduction:add2}
			\State $a \gets \Call{SSRR}{\tilde{D}}$ \Comment{$a \in \RiemannRoch(D + gA)$} \label{alg:hess_reduction:ssrr2}
			\State \Return $g, a$ \label{alg:hess_reduction:ret1}
		\EndIf
		\For{$m \gets g - 2, \dots, 0$}
			\State $\tilde{D} \gets \tilde{D} - A$ \Comment{$\tilde{D} = D + mA$} \label{alg:hess_reduction:add_loop}
			\State $a' \gets \Call{SSRR}{\tilde{D}}$ \label{alg:hess_reduction:ssrr_loop}
			\If{$a' = \texttt{None}$} \Comment{$\RiemannRochDim(D + (m + 1) A) = 1$} \label{alg:hess_reduction:loop_break_condition}
			\State \Return $m + 1, a$ \Comment{$a \in \RiemannRoch(D + (m + 1) A)$} \label{alg:hess_reduction:ret2}
			\EndIf
			\State $a \gets a'$
		\EndFor
		\State \Return $0, a$ \label{alg:hess_reduction:ret3}
	\end{algorithmic}
\end{algorithm}

\begin{theorem} \label{thm:Correctness-of-reduction}
	Let $F/K$ be a function field of genus $g \geq 2$.
	Given $D \in \DivZero(F)$, \cref{alg:hess_reduction} returns $(r, a)$ where $r$ solves \texttt{HR-Min} for $D$ and $0 \neq a \in \RiemannRoch(D + rA)$.
\end{theorem}
\begin{proof}
	Let $D \in \DivZero(F)$. We show that \cref{alg:hess_reduction:ret1,alg:hess_reduction:ret2,alg:hess_reduction:ret3} in \cref{alg:hess_reduction} each returns the correct pair $(r,a)$.

	First, assume that \cref{alg:hess_reduction} returns from \cref{alg:hess_reduction:ret1}.
    Then $\RiemannRochDim(D + (g - 1) A) = 0$.
	By \cite[Proposition 8.2]{hessComputingRiemannRochSpaces2002}, \texttt{HR-Min} has a solution, and by the monotonicity of Riemann-Roch spaces (\cref{lem:riemann_roch_monotonicity}),
    we must have $\RiemannRochDim(D + gA) = 1$ in order for such a solution to exist.
	Therefore $r = g$ is the solution to \texttt{HR-Min}, and since $a \in \RiemannRoch(D + gA)$, the output~$(r,a)$ of \cref{alg:hess_reduction} is correct.

	Next, assume that \cref{alg:hess_reduction} returns from  \cref{alg:hess_reduction:ret2}.
	In order for this line to be reached, we must have $\RiemannRochDim(D + (g - 1) A) \geq 1$.
	By monotonicity $\RiemannRochDim(D + gA) \geq 1$, so $r=g$ does not solve \texttt{HR-Min}.
	Let $m$ be the value of the loop variable at the point where \cref{alg:hess_reduction} reaches \cref{alg:hess_reduction:ret2}.
	Then $\RiemannRochDim(D + mA) = 0$.
	Then by monotonicity, $\RiemannRochDim(D + m'A) = 0$ for all $m' < m$.
	Since the condition in \cref{alg:hess_reduction:loop_break_condition} is not satisfied for $m + 1$, we have $\RiemannRochDim(D + (m+1)A) \geq 1$.
	By monotonicity and the existence of a minimal solution, $r = m+1$ is the minimal integer that solves \texttt{HR-Min}.  
	As $a \in \RiemannRoch(D + (m + 1) A)$, \cref{alg:hess_reduction} returns the correct value on \cref{alg:hess_reduction:ret2}.

	Finally suppose that \cref{alg:hess_reduction} returns from \cref{alg:hess_reduction:ret3}.
	For this point to be reached we must have $\RiemannRochDim(D + mA) \geq 0$ for all $0 \leq m \leq g$.
	Since a solution to \texttt{HR-Min} with $0 \leq r \leq g$ exists, it follows that 
    $r = 0$ must be the solution.
	Accordingly, our algorithm returns $r=0$ as the solution to \texttt{HR-Min} in this case.
	Furthermore, since $g \geq 2$ by assumption, the for loop will run at least once, so we have $a \in \RiemannRoch(D)$.
    \qed
\end{proof}

In the most likely case where $r = g$ solves \texttt{HR-Min}, \cref{alg:hess_reduction} completes with only two invocations of $\SSRR$, in contrast to a binary search approach which requires $\bigO(\log g)$ such calls.
We also note that the algorithm handles the second most likely case where $r = g - 1$ is the solution with the same number of steps as the $r = g$ case, so the optimal scenario occurs with expected probability $1 - q^{-2}$.

Our main strategy for designing \cref{alg:hess_reduction} was to minimize the number of costly calls to $\SSRR$. The remainder of the work is given by the divisor additions in  \cref{alg:hess_reduction:add1,alg:hess_reduction:add2,alg:hess_reduction:add_loop}.
In practice, software implementations such as SageMath \cite{sagemath} represent divisors as pairs of fractional ideals of the finite and infinite maximal orders, as suggested in \cite{hessComputingRiemannRochSpaces2002}.
With this representation, we are able to effect an additional performance gain at the cost of a reasonable increase to memory usage by exploiting caching, which we will describe next.

\subsection{Caching}
We now describe how caching can be used to speed up computation of Hess-reductions, at a modest increase in memory cost, when $A$ is chosen to be a degree one infinite place.
Although function fields $F/K(x)$ of degree $n$ have at most $n$ infinite places (as opposed to infinitely many finite places), about half of the work of traditional generic Jacobian arithmetic entails  arithmetic on infinite places, which suggest an opportunity to make use of caching.
We note that caching intermediate results involving the infinite places still works if $A$ is chosen to be a finite place of degree one, but there are more opportunities to use cached values when $A$ is a degree one infinite place.
For instance, if $A$ is a degree one infinite place, then can used cached values in \cref{alg:hess_reduction:add1,alg:hess_reduction:add2,alg:hess_reduction:add_loop}.

For function fields with exactly one infinite place of degree one (such as superelliptic and in particular ramified hyperelliptic fields), every degree zero divisor is of the form $D = D^0 + D^\infty$ with $D^\infty = -\deg(D)\infty$, where $\infty$ denotes the infinite place. So  arithmetic involving $D$ is entirely governed by arithmetic on $D^0$. Therefore, we henceforth assume that $F / K$ is a function field of genus~$g$ with exactly $t \geq 2$ infinite places, denoted $\infty_1, \dots, \infty_t$, where we assume that $\deg(\infty_t) = 1$.
Let $D_1, D_2 \in \DivZero(F)$ be Hess-reduced along $\infty_t$, so $D_1 = \tilde{D_1} - r_1 \infty_t$ and $D_2 = \tilde{D_2} - r_2 \infty_t$ with $0 \leq r_1, r_2 \leq g$.
Let $D_3 = D_1 + D_2$.
If divisors are represented as pairs consisting of their finite and infinite parts, then we must compute
$\FinitePart(D_3) = \FinitePart(D_1) + \FinitePart(D_2)$ and $\InfinitePart(D_3) = \InfinitePart(D_1) + \InfinitePart(D_2)$.
We call the process of computing either $\FinitePart(D_3)$ or $\InfinitePart(D_3)$ a \defemph{partial divisor addition}.
Because $D_1$ and $D_2$ are Hess-reduced, the number of possibilities for $\InfinitePart(D_3)$ are effectively bounded, suggesting that caching is worthwhile.
A lengthy but elementary combinatorial derivation produces a loose upper bound on the cache for sums of infinite parts of Hess-reduced divisors that is in $\bigO(g^{t - 1})$ ideals of $\InfiniteMaximalOrder$.
A similar technique can be applied to one of the steps involved in performing $\SSRR$, whose a memory cost is $\bigO(g^{2t + 1})$.
For $\SSRR$, we are caching a map from infinite ideals to $n \times n$ matrices whose entries are rational functions over $K$ of numerator and denominator degree at most $g$.
See \cite[Subsection 4.5.2]{macriMScThesis2025} for the details of this combinatorial calculation.

We note that empirically, the actual number of cached elements is substantially smaller than these bounds.
In our experiments, the size of these caches was never an issue, and was inconsequential over larger finite fields.
The typical case where $r = g$ solves \texttt{HR-Min} happens less frequently for smaller finite fields, causing more distinct intermediate values to be cached when working over small fields like $\FF_2$ or $\FF_3$.
For $g = 25$, $t = 2$, and $\FF_{32\,771}$, after approximately 5000 Jacobian additions, the observed cache size for the addition of the infinite parts of divisors was 60, and for the $\SSRR$ cache it was 33.

\section{Time Complexity of Computing Hess-Reductions}
\label{sec:complexity}
We now present a time complexity analysis of \cref{alg:hess_reduction}, with and without caching, and compare it to the time complexity for the binary search approach suggested in \cite{hessComputingRiemannRochSpaces2002}.
To ensure a fair comparison, we also give the time complexity for the binary search with and without using our caching technique.
We do not consider the aforementioned impact of caching on the subroutine $\SSRR$, as it only speeds up one small part of the $\SSRR$ computation.

The height of a divisor provides an appropriate measure of its size, and will be used in our complexity statements.
\begin{definition}
	\citeresult{hessComputingRiemannRochSpaces2002}{p.~434}
	Let $D \in \Div(F)$.
	We define the \defemph{(divisor) height} of $D$ to be:
	\[\DivHeight(D) \coloneqq \sum_{P \in \Places(F)} \abs{v_P(D)} \cdot \deg(P).\]
\end{definition}

Note that $\DivHeight(D) = \deg(D)$ when $D \ge 0$. We will also express our running times in terms of the costs of Riemann-Roch basis computations and divisor additions, using the following notation:
\begin{definition}
	\citeresult{macriMScThesis2025}{Notation 4.3.3}
	Denote by $\SSRRComplexity(h)$ the time complexity, expressed in field operations in $K$, of a short-circuited Riemann-Roch computation on input a divisor $D \in \Div(F)$ with $h = \DivHeight(D)$.

	Let $D_1, D_2 \in \Div^0(F)$ be Hess-reduced.
	Denote by $\PartialAdditionComplexity$ the time complexity, expressed in field operations in $K$, of computing either $\FinitePart(D_3)$ or $\InfinitePart(D_3)$, where $D_3 = D_1 + D_2$.
\end{definition}

\subsection{Time Complexity in High-Level Operations}
\begin{theorem}
	\label{thm:linear_complexity}
	Let $D = \bar{D} - \bar{r}A$ be the unreduced sum of two Hess-reduced divisors of $F$.
	Assume that $(g - 1) A$ and $-A$ have both been precomputed.
	Then the following hold:
	\begin{enumerate}
		\item In the worst case, the complexity of \cref{alg:hess_reduction} on input $D$ is:
			\begin{equation}
				\label{eq:complexity_operations:worst}
				g \PartialAdditionComplexity + \sum_{m = 0}^{g - 1} \SSRRComplexity(\DivHeight(D + mA)).
			\end{equation}
		\item In the heuristically average case, the complexity of \cref{alg:hess_reduction} on input $D$ is:
			\begin{equation}
				\label{eq:complexity_operations:average}
				2 \PartialAdditionComplexity + \SSRRComplexity(3g + 1) + \SSRRComplexity(3g).
			\end{equation}
	\end{enumerate}
\end{theorem}
\begin{proof}
	The worst case complexity follows from the design of the algorithm.

	For the heuristically average case, let $D_1$ and $D_2$ be the two Hess-reduced divisors such that $D = D_1 + D_2$.
	Assume that $r = g$ solves \texttt{HR-Min}, and that $D_1 = \tilde{D}_1 - r_1 A$ and $D_2 = \tilde{D}_2 - r_2 A$ are also both typical divisors, hence $r_1 = r_2 = g$.
	We may write $\bar{D} = \tilde{D}_1 + \tilde{D}_2$ and $\bar{r} = 2g$, so $\deg(\bar{D}) = 2g$.
    \cref{alg:hess_reduction} computes the divisor sums $D + (g - 1) A$ and $D + (g-1)A + A = D + gA$. Noting that $\bar{D} \geq 0$ and that $A \notin \supp(\bar{D})$ by \cref{prop:Hess_Dtilde_A_free}, we compute the heights of these divisors to be
	\begin{align*}
		\DivHeight(D + (g - 1)A) &= \DivHeight(\bar{D}) + \DivHeight(-2g A + (g - 1) A) = 3g + 1,\\
		\DivHeight(D + gA) &= \DivHeight(\bar{D}) + \DivHeight(- 2g A + gA) = 3g,
	\end{align*}
	whence the heuristically average case complexity follows. \qed
\end{proof}

For completeness, we give the analogous result for computing Hess-reductions using binary search for solving \texttt{HR-Min}.
In the typical case where $r_1 = r_2 = r = g$, this complexity is worse than that of \cref{thm:linear_complexity} for $g \geq 4$.

\begin{theorem}
	\label{thm:binary_complexity}
	Let $D = \bar{D} - \bar{r}A$ be the unreduced sum of two Hess-reduced divisors.
	Assume that $mA$ has been precomputed for all $0 \leq m \leq g$.
	Then both the worst case and heuristically average case complexity of finding the solution $r$ to \texttt{HR-Min} on input $D$ with binary search and computing $a \in \RiemannRoch(D + rA)$ is
	\[ \ceil{\log_2(g + 1)} \PartialAdditionComplexity + \sum_{i=1}^{\ceil{\log_2(g + 1)}} \SSRRComplexity(\DivHeight(D + \left \lceil \frac{g(2^i - 1)}{2^i} \right \rceil A)).\]
\end{theorem}
\begin{proof}
	The worst case for binary search is when the value being searched for is at an endpoint of the search range, which it is in the heuristically average case where $r = g$ solves \texttt{HR-Min}.
	Here binary search needs to compute $D + mA$ and $\SSRR(D + mA)$ for all $m = \ceil*{g \cdot \frac{2^i - 1}{2^i}}$ for all $1 \leq i \leq \ceil{\log_2(g + 1)}$. \qed
\end{proof}

\subsection{Time Complexity in Base Field Operations}
To give the time complexity in field operations, we must first fix a representation of divisors, then substitute the best known values for the two subroutines $\PartialAdditionComplexity$ and $\SSRRComplexity(h)$ into \cref{thm:linear_complexity,thm:binary_complexity}.

In practice, divisors are represented by a pair of ideals $(I, J)$ where $I$ is an ideal of the finite maximal order $\FiniteMaximalOrder$ and $J$ is an ideal of the infinite maximal order $\InfiniteMaximalOrder$.
These ideals are themselves represented as $n \times n$ matrices in Hermite Normal Form whose entries are polynomials in $K[x]$ of degree at most $g$ when the divisors are Hess-reduced.
The complexity of $\PartialAdditionComplexity$ is simply the complexity of multiplying two ideals and converting the resulting matrix to Hermite Normal Form.

The best asymptotic complexity for $\PartialAdditionComplexity$ that we are aware of arises from replacing the naive polynomial multiplication with FFT in \cite[Theorem 5.2.2]{tangInfrastructureFunctionFields2011}, yielding:
\begin{equation}
	\label{eq:partial_addition_complexity}
	\PartialAdditionComplexity = \bigO(\EqPartialAdditionComplexity).
\end{equation}
The best asymptotic complexity for $\SSRRComplexity(h)$ that we are aware of can be found in \cite[Corollary 4.1.5]{bauchLatticesPolynomialRings2014} and is given by
\begin{equation}
	\label{eq:ssrr_addition_complexity}
	\SSRRComplexity(h) = \bigO(\EqSSRRComplexity).
\end{equation}
In practice, the performance is better than predicted by \cref{eq:partial_addition_complexity,eq:ssrr_addition_complexity}, as seen in Figures \ref{fig:reduction_impact:n3_p32771_g} and \ref{fig:reduction_impact:g15_p32771_n}.

Jacobian arithmetic with Hess-reduced divisors is done using \cref{alg:jacobian_arithmetic}. So the complexity of Jacobian arithmetic is simply the complexity of \cref{alg:jacobian_arithmetic} with whatever algorithm is used to compute $(r, a)$ in \cref{alg:jacobian_arithmetic:hr_min}.
\begin{algorithm}
	\caption{Jacobian arithmetic with Hess-reduced divisors}
	\label{alg:jacobian_arithmetic}
	\begin{algorithmic}[1]
		\Require Hess-reduced divisors $D_1 = \tilde{D}_1 - r_1 A$ and $D_2 = \tilde{D}_2 - r_2 A$
		\Ensure $D_3 = \tilde{D}_3 - r_3 A$ with $D_3$ Hess-reduced and $D_3 \equiv D_1 + D_2$

		\State $E \gets D_1 + D_2$
		\State \label{alg:jacobian_arithmetic:hr_min} Compute $(r, a)$ where $r$ is a solution to \texttt{HR-Min} for input $E$ and $a \in \RiemannRoch(D + rA) \setminus \set{0}$ 
		\State $D_3 \gets E + \principaldivisor{a}$
	\end{algorithmic}
\end{algorithm}

We can now give the time complexity of Jacobian arithmetic with \cref{alg:hess_reduction} or with binary search, both with and without caching.
We account for the impact of caching by assuming that adding the infinite parts of divisors is amortized as $\bigO(1)$.

\begin{theorem} \label{thm:complexities}
	Computing the Hess-reduction of two Hess-reduced divisors of $F$ has the following asymptotic time complexity, expressed in operations in $K$.
	\begin{enumerate}
		\item Using \cref{alg:hess_reduction} in \cref{alg:jacobian_arithmetic:hr_min} of \cref{alg:jacobian_arithmetic} and no caching, 
        the worst case time complexity is
            \[\bigO(\EqUniqueHessAdditionLinComplexityWorst)\]
        and the heuristically average case complexity is
			\[\bigO(\EqUniqueHessAdditionLinComplexity)\]
		\item Using \cref{alg:hess_reduction} in \cref{alg:jacobian_arithmetic:hr_min} of \cref{alg:jacobian_arithmetic} and caching, the worst case time complexity is
        \[\bigO(\EqUniqueHessAdditionLinComplexityCachingWorst)\]
        and the heuristically average case complexity is
    		\[\bigO(\EqUniqueHessAdditionLinComplexityCaching)\]
		\item Using binary search in \cref{alg:jacobian_arithmetic:hr_min} of \cref{alg:jacobian_arithmetic} and no caching, the time complexity is
			\[\bigO(\EqUniqueHessAdditionBSComplexity)\]
		\item Using binary search in \cref{alg:jacobian_arithmetic:hr_min} of \cref{alg:jacobian_arithmetic} and caching, the time complexity is
			\[\bigO(\EqUniqueHessAdditionBSComplexityCaching)\]
	\end{enumerate}
\end{theorem}
\begin{proof}
	In \cref{alg:jacobian_arithmetic}, we add the finite and infinite parts of $D_1$ and $D_2$ to compute $E$, for a cost of $2 \PartialAdditionComplexity$.
	One of these $\PartialAdditionComplexity$ operations is on the infinite parts of $D_1$ and $D_2$, hence amortized as $\bigO(1)$ if caching is used.
	We then use either \cref{alg:hess_reduction} or binary search to find $(r, a)$, with the corresponding costs given in \cref{thm:linear_complexity,thm:binary_complexity}
	Finally, we compute $D_3 = E + \principaldivisor{a}$ for a cost of $2 \PartialAdditionComplexity$, and again one of these is $\bigO(1)$ if caching is used.
	The computation of $\principaldivisor{a}$ from~$a$ is dominated by the other costs in the algorithm, so it can be ignored in the asymptotic analysis.

	After accounting for caching in \cref{alg:hess_reduction} and in the binary search approach and substituting in the values from \cref{eq:partial_addition_complexity,eq:ssrr_addition_complexity} into \cref{thm:linear_complexity,thm:binary_complexity} we get the stated values. \qed
\end{proof}

We determine the overall asymptotic complexities when one of $g, n$ is fixed and the other varies. Although caching makes a difference in practical performance, after fixing either $g$ or $n$ and allowing the other to vary (see \cref{fig:reduction_impact:n3_p32771_g,fig:reduction_impact:g15_p32771_n}), it only impacts lower order terms in the overall complexities. 

\begin{corollary} \label{cor:complexities}
    Computing the Hess-reduction of two Hess-reduced divisors of $F$ has the following asymptotic time complexity, expressed in operations in $K$.
    \begin{enumerate} \itemsep 1pt
		\item[] When $g$ is fixed and $n \to \infty$: $O(C_f^2 n^9)$ 
        \item[] When $n$ is fixed and $g \to \infty$:
        \begin{enumerate} \itemsep 1pt
            \item $O(C_f^2 g)$ in the worst case when using \cref{alg:hess_reduction} in \cref{alg:jacobian_arithmetic:hr_min} of \cref{alg:jacobian_arithmetic};
            \item $O(C_f^2)$ heuristiclly on average when using \cref{alg:hess_reduction} in \cref{alg:jacobian_arithmetic:hr_min} of \cref{alg:jacobian_arithmetic};
            \item $O(C_f^2 \log(g))$ when using binary search in \cref{alg:jacobian_arithmetic:hr_min} of \cref{alg:jacobian_arithmetic}.
        \end{enumerate}
    \end{enumerate}
\end{corollary}
\begin{proof}
    For the case of fixed $g$ and varying $n$, the term $C_f^2n^9$ dominates in all the expressions of \cref{thm:complexities}. For the case of fixed $n$ and varying $g$, we use the identity
    \begin{equation} \label{eq:genus_bound}
    g \leq \frac{1}{2}(C_fn-2)(n-1);
    \end{equation}
    established in \cite[Cor.\ 5.6]{hessComputingRiemannRochSpaces2002} and identify the dominant term in each of the complexity estimates of \cref{thm:complexities}.
    \qed
\end{proof}

Asymptotically, for fixed $n$ and varying $g$, the heuristic average complexity of using linear search in \cref{alg:jacobian_arithmetic} differs from the complexity of using binary search by a factor of $O(\log(g))$. Indeed, for typical Hess-reduced input divisors, \cref{alg:hess_reduction} will complete after $2$ reductions, wheres binary search is expected to perform $\ceil{\log_2(g + 1)}$ reductions, yielding a relative factor of $\ceil{\log_2(g + 1)}/2$.

\section{Timing Experiments}
\label{sec:timing_experiments}
We have implemented Jacobian arithmetic on Unique Hess representations in SageMath using both \cref{alg:hess_reduction} and binary search for reduction, both with and without caching, and profiled the results.
Our implementation of Unique Hess arithmetic using \cref{alg:hess_reduction} was recently accepted into SageMath, available since the SageMath 10.9.beta7 release \cite{sagemath}.
Our implementation of the binary search approach as well as our code used to profile the implementations is available at \url{https://github.com/vincentmacri/Improvements-to-Jacobian-Arithmetic-in-Global-Function-Fields}
Our implementations of \cref{alg:hess_reduction} and the binary search variant were written in Python.
We also optimized the performance-critical short-circuited Riemann-Roch computations by converting the previous Python implementation in SageMath to a Cython implementation.

We ran our timing experiments on a server running Red Hat Enterprise Linux 9.7 with an Intel Xeon CPU E7-8891 v4 with 80 64-bit CPU cores at 2.80 GHz, and 256 GiB of memory.
Our experiments were run on a custom fork of Sage~10.7.beta7, using Python~3.12.5.
SageMath has over 100 software dependencies, so we will only list major relevant dependencies: Singular~4.4.1 \cite{singular}, FLINT~3.4.0 \cite{flint}, GMP~6.3.0 \cite{gmplib}, and numpy~2.3.2 \cite{numpy}.
We also used Magma~2.27-6 \cite{magma} to perform some genus computations that took too long in SageMath.
Our use of Magma was restricted to computing genera when initializing some function fields, and so it has no impact on our reported timing results.

\subsection{Experimental Setup}
For testing, we generated global function fields both via an ad-hoc method and via the method described in \cite{tangInfrastructureFunctionFields2011}.
The latter provides an easy way to generate global function fields with two infinite places of degree one.
These function fields are given by a defining polynomial $f(t) \in K[x, t]$ of degree $n = [F : K(x)]$, irreducible over $K(x)$, and of the form:
\[f(t) = t^n + \sum_{i=0}^{n - 1} a_i t^i,\]
with $a_i \in K[x]$ for $0 \le i \le n-1$, and
\begin{align*}
	\deg(a_{n - 1}) &= C_f,\\
	\deg(a_i) &< (n - i) C_f \quad \text{for $1 \le i \le n-1$},\\
	\deg(a_0) &= nm - 1.
\end{align*}
We conjecture that for these function fields, equality always holds in \eqref{eq:genus_bound}. We verified this for all the fields we tested in \cref{fig:reduction_impact:n3_p32771_g}, so our timing results do not depend on this assumption. Note that equality in \eqref{eq:genus_bound} implies that any two of $g$, $n$ and $C_f$ uniquely determine the third.

For \cref{fig:reduction_impact:g15_p32771_n}, we instead generated function fields via an ad-hoc method, by trying random polynomials $f(x,t)$, monic and of degree 3 in $t$, with appropriate bounds on $C_f$, verifying that they were irreducible over $K(x)$ and that the function field contained a degree one infinite place. Equality did not hold in \cref{eq:genus_bound} for the function fields generated by this ad-hoc method.

All reported timing tests were performed over the finite field $\FF_{32771}$, a 16 bit prime field.
Preliminary testing over larger prime fields indicated that the same observed trends continue to hold in this setting, but with additional overhead incurred by having to resort to 64-bit arithmetic.
For each triple $(g, n, p)$, we performed tests in 5 function fields.
In each function field we performed Unique Hess arithmetic on 5 Fibonacci-style addition chains of length 1000.
Reported results are the average over all function fields and addition chains for each $(g, n, p)$.

To test the correctness of our implementations, we relied on SageMath's robust testing framework that is able to automatically create a test suite and check correctness by verifying that the abelian group axioms hold on a set of examples.

\subsection{Timing Results}
Overall, \cref{alg:hess_reduction} significantly outperformed the binary search approach.
In our figures, we plotted the average CPU time in milliseconds to perform a single Jacobian addition for various parameters and implementations, alongside the asymptotic complexity.
For \cref{alg:hess_reduction} we used the heuristically average case complexity.
For the asymptotic complexity plots in \cref{fig:reduction_impact:n3_p32771_g}, we assumed equality in \cref{eq:genus_bound} and replaced $C_f$ accordingly.
In \cref{fig:reduction_impact:g15_p32771_n} we treated $C_f$ as a constant because $C_f^2$ is negligible compared to $n^9$ asymptotically, and the largest value of $C_f$ in our data plotted in \cref{fig:reduction_impact:g15_p32771_n} was $C_f = 6$.
To account for constant factors in the plots of our asymptotic estimates, we computed a constant $a$ such that if the complexity is $\bigO(f(x))$, then $a f(x)$ matches the data at a single measured value.
For \cref{fig:reduction_impact:n3_p32771_g}, since our improvement is by a factor of $\bigO(\log(g))$, we matched the constant $a$ to the largest data point $g = 55$ in order to more accurately capture the dependence on $g$.
In \cref{fig:reduction_impact:g15_p32771_n}, we matched the constant $a$ to the first data point as our improvements do not affect the asymptotic complexity in $n$.

In \cref{fig:reduction_impact:n3_p32771_g}, we see that caching yields a noticeable performance gain, but the most significant gains come from using \cref{alg:hess_reduction}.
In particular, for $g \geq 13$, the non-caching implementation using \cref{alg:hess_reduction} outperforms the caching implementation using binary search.
As predicted by our complexity analysis, the difference in performance between \cref{alg:hess_reduction} and the binary search approach grows with $g$.
Our data show an increase in speed that is slightly less than the predicted factor of $\ceil{\log(g+1)}/2$ for linear versus binary search, likely because this speed-up
only applies in the most expensive part of Jacobian arithmetic (reduction to a unique representative), and the lower degree terms that were excluded from the analysis in \Cref{cor:complexities} are not entirely negligible for our parameter range.
For genus 100 (not depicted in \cref{fig:reduction_impact:n3_p32771_g}), the linear search variant was about $3$ times faster both with and without caching. Here, we would expect the performance to increase by a factor of about $\ceil{\log_2(101)} / 2 = 3.5$, which is reasonably close to our measured speed-up of a factor of $3$.

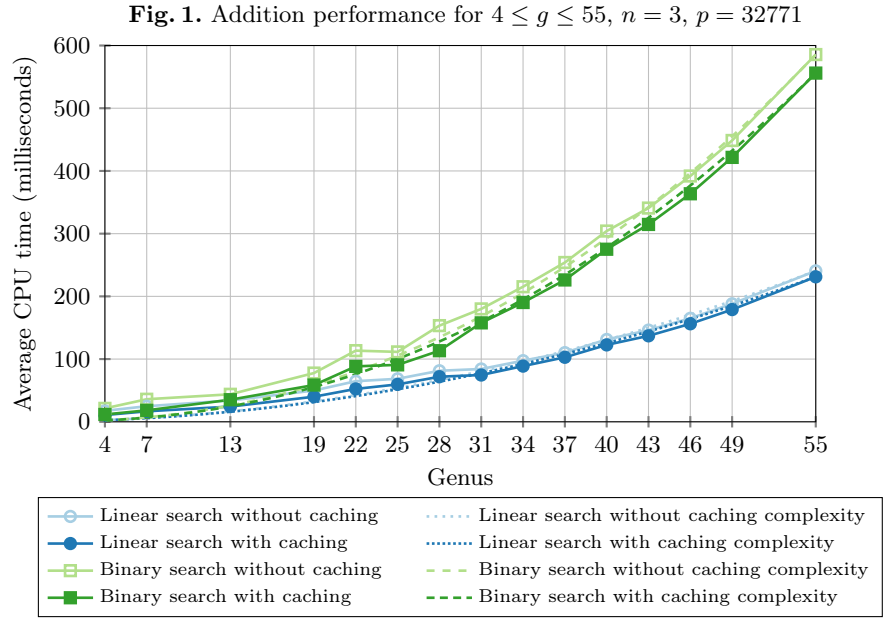
\begin{figure}
\caption{Addition performance for $4 \leq g \leq 55$, $n = 3$, $p = 32771$}
\label{fig:reduction_impact:n3_p32771_g}
\begin{tikzpicture}
\begin{axis}[width=0.9\textwidth,height=0.34\textheight,xlabel={Genus},xmin={4},xmax={55},ytick distance=100,ymin=0,ymax=600,restrict y to domain=-100:1200,ylabel={Average CPU time (milliseconds)},legend columns={2},xtick=data]
\addplot [mark=o, index of colormap=0 of Paired] table [x=genus, y=linear_no_caching_milliseconds_per_addition] {unique_n3_p32771_g.dat};
\addlegendentry {Linear search without caching}
\addplot [index of colormap=0 of Paired, domain=4:55, dotted, mark repeat=5, mark phase=1] {(x + 2)^(2)*(1781014361/24044978268)};
\addlegendentry {Linear search without caching complexity}
\addplot [mark=*, index of colormap=1 of Paired] table [x=genus, y=linear_caching_milliseconds_per_addition] {unique_n3_p32771_g.dat};
\addlegendentry {Linear search with caching}
\addplot [index of colormap=1 of Paired, domain=4:55, densely dotted, mark repeat=5, mark phase=2] {(x + 2)^(2)*(695604737/9776578896)};
\addlegendentry {Linear search with caching complexity}
\addplot [mark=square, index of colormap=2 of Paired] table [x=genus, y=binary_no_caching_milliseconds_per_addition] {unique_n3_p32771_g.dat};
\addlegendentry {Binary search without caching}
\addplot [index of colormap=2 of Paired, domain=4:55, dashed, mark repeat=5, mark phase=3] {(x + 2)^(2)*(log2(x))*(550460255827807/17656975986636588)};
\addlegendentry {Binary search without caching complexity}
\addplot [mark=square*, index of colormap=3 of Paired] table [x=genus, y=binary_caching_milliseconds_per_addition] {unique_n3_p32771_g.dat};
\addlegendentry {Binary search with caching}
\addplot [index of colormap=3 of Paired, domain=4:55, densely dashed, mark repeat=5, mark phase=4] {(x + 2)^(2)*(log2(x))*(18618848677700/628949287680379)};
\addlegendentry {Binary search with caching complexity}
\end{axis}
\end{tikzpicture}
\end{figure}

For fixed $g$ and varying $n$, all our algorithms have same the asymptotic complexity by \cref{cor:complexities}.
With $g$ fixed, caching effects a speed-up by a constant factor in the number of partial divisor additions, with a more pronounced performance gain for the computationally more intense binary search compared to \cref{alg:hess_reduction}.  However, this constant factor speed-up is dominated by the cost of the short-circuited Riemann-Roch computations.
Linear search is expected to be faster on average than binary search by a constant factor of approximately $\ceil{\log_2(g + 1)} / 2$ as explained earlier; for $g = 15$, this factor is $2$. The data depicted in \cref{fig:reduction_impact:g15_p32771_n} bear this out and also suggest that the asymptotic upper bound of $O(C_f^2 n^9)$ is an over-estimate of the actual complexity of Jacobian arithmetic.

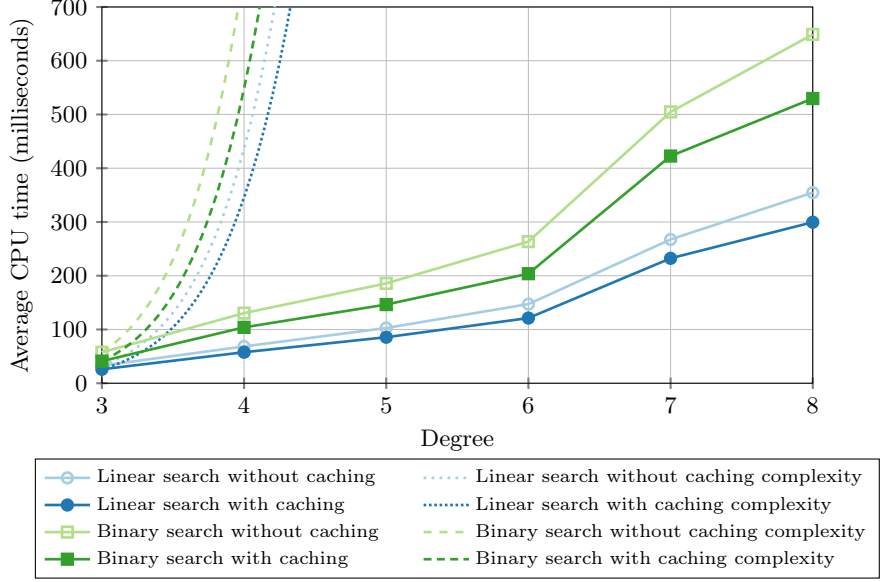
\begin{figure}
\caption{Addition performance for $3 \leq n \leq 8$, $g = 15$, $p = 32771$}
\label{fig:reduction_impact:g15_p32771_n}
\begin{tikzpicture}
\begin{axis}[width=0.9\textwidth,height=0.34\textheight,xlabel={Degree},xmin={3},xmax={8},ytick distance=100,ymin=0,ymax=700,restrict y to domain=-100:1400,ylabel={Average CPU time (milliseconds)},legend columns={2},xtick=data]
\addplot [mark=o, index of colormap=0 of Paired] table [x=degree, y=linear_no_caching_milliseconds_per_addition] {unique_g15_p32771_n.dat};
\addlegendentry {Linear search without caching}
\addplot [index of colormap=0 of Paired, domain=3:8, dotted, mark repeat=5, mark phase=1] {(x)^(9)*(391106297/234915010677)};
\addlegendentry {Linear search without caching complexity}
\addplot [mark=*, index of colormap=1 of Paired] table [x=degree, y=linear_caching_milliseconds_per_addition] {unique_g15_p32771_n.dat};
\addlegendentry {Linear search with caching}
\addplot [index of colormap=1 of Paired, domain=3:8, densely dotted, mark repeat=5, mark phase=2] {(x)^(9)*(3013114141/2284599885417)};
\addlegendentry {Linear search with caching complexity}
\addplot [mark=square, index of colormap=2 of Paired] table [x=degree, y=binary_no_caching_milliseconds_per_addition] {unique_g15_p32771_n.dat};
\addlegendentry {Binary search without caching}
\addplot [index of colormap=2 of Paired, domain=3:8, dashed, mark repeat=5, mark phase=3] {(x)^(9)*(275160865/94151294784)};
\addlegendentry {Binary search without caching complexity}
\addplot [mark=square*, index of colormap=3 of Paired] table [x=degree, y=binary_caching_milliseconds_per_addition] {unique_g15_p32771_n.dat};
\addlegendentry {Binary search with caching}
\addplot [index of colormap=3 of Paired, domain=3:8, densely dashed, mark repeat=5, mark phase=4] {(x)^(9)*(30164135/14422141611)};
\addlegendentry {Binary search with caching complexity}
\end{axis}
\end{tikzpicture}
\end{figure}

\section{Conclusions and Further Research}
\label{sec:conclusions}
Our reduction algorithm (\Cref{alg:hess_reduction}) was designed specifically for optimal performance in the most frequent cases.
Further speed-ups will likely require improvements to the underlying 
subroutines, namely Riemann-Roch space computation and ideal multiplication.
 Additional information about the input divisor might potentially narrow down the possible candidate solutions for \texttt{HR-Min}, but given the overwhelming likelihood of the generic case, it is difficult to see how that approach would lead to a significant performance gain in this situation. Even \emph{a priori} ruling out the case where $r = g - 1$ solves \texttt{HR-Min} without performing a short-circuited Riemann-Roch computation would only speed up the algorithm by a constant factor.

The timing experiments of \cref{sec:timing_experiments} demonstrate that \cref{alg:hess_reduction} is the fastest unique divisor reduction algorithm in practice, at least for the function fields we tested.
As the non-caching version of the linear search implementation outperformed the caching version of the binary search implementation, we do not need to assume the existence of a degree one infinite place for the linear search method to outperform the binary search approach.
Our timing experiments also demonstrate that caching computations involving the infinite parts of divisors gives a noticeable speedup when the function field has a degree one infinite place.

\subsection{Future Work}
We list several possible directions for future work.

For short-circuited Riemann-Roch computations, we used SageMath's modified version of the algorithm for computing $K$-bases of Riemann-Roch spaces from \cite{hessComputingRiemannRochSpaces2002}.
\Cref{alg:hess_reduction} demonstrates that short-circuited Riemann-Roch computations are sufficient for some applications.
It would be beneficial to have a potential short-circuited Riemann-Roch computation approach that is faster than simply exiting early from an algorithm designed to perform a full Riemann-Roch basis computation.
    
For hyperelliptic curves, the fastest known algorithm for Jacobian arithmetic is NUCOMP \cite{JacobsonVanderPoortenNUCOMP,lindnerBalancedNUCOMP2020}. The traditional approach to divisor arithmetic first computes the sum of two input divisors and subsequently reduces the result. NUCOMP in essence interleaves divisor addition and reduction, thereby operating on smaller inputs and avoiding the costly computation of bases for non-reduced intermediate divisors. It is worthwhile to explore whether a similar strategy can be realized for Jacobian arithmetic for non-hyperelliptic curves. 
    
The first framework for arithmetic on unique representatives of elements in the Jacobian of a split hyperelliptic curve (i.e.\ a hyperelliptic curve with two infinite places) is due to Paulus and R\"uck \cite{paulusRealImaginaryQuadratic1999}. In addition to reduction, the Paulus-R\"uck arithmetic required additional adjustment steps to obtain unique representatives of Jacobian elements. These adjustment steps were eliminated in the balanced divisor arithmetic introduced by Galbraith, Harrison, and Mireles Morales \cite{mirelesmoralesEfficientArithmeticHyperelliptic2008,galbraithEfficientHyperellipticArithmetic2008,galbraithMathematicsPublicKey2018}. In \cite{macriMScThesis2025}, it was shown that unique reduced divisor class representatives used \cite{paulusRealImaginaryQuadratic1999} are a special case of the Hess-reductions introduced in \cref{def:hess_reduced} when the curve is split hyperelliptic. This relationship could potentially be leveraged to explore whether the notion of a balanced divisor could be extended to non-hyperelliptic curves with two infinite places, or even to curves with more infinite places, with the aim to simplify Jacobian arithmetic and/or improve its performance.

    The notion of a semi-reduced divisor plays a crucial role in Jacobian arithmetic for hyperelliptic curves. Divisor addition produces a semi-reduced divisor that is equivalent to the sum of two input divisors, which is subsequently reduced. The concept of semi-reduced divisors was generalized to arbitrary function fields in  \cite{macriMScThesis2025}; in essence, a divisor $D$ is semi-reduced if no sub-sum of $D$ is the conorm of a divisor of $K(x)$. The author of \cite{macriMScThesis2025} also proved that Hess-reductions along an infinite place of degree one are semi-reduced.  
    Future work could explore whether the notion of semi-reducedness can be generalized further in such a way that Hess-reductions along a finite degree one place, or even along a degree one divisor as considered in \cite{hessComputingRiemannRochSpaces2002}, are semi-reduced. An appropriate notion of semi-reducedness may be  
    a key ingredient for an efficient reduction algorithm that produces unique representatives of divisor classes.

\begin{credits}

\subsubsection{\ackname}
All three authors received funding from the Natural Sciences and Engineering Research Council of Canada.
The first author also received funding from the Province of Alberta through an AGES Research scholarship.

We also thank Kwankyu Lee for writing much of the SageMath function field machinery that we built on, and for reviewing the first author's pull request \url{https://github.com/sagemath/sage/pull/41453} adding Unique Hess Jacobian arithmetic to SageMath.

\subsubsection{\discintname}
The authors have no competing interests.
\end{credits}

\clearpage
\bibliographystyle{splncs04}
\bibliography{References}
\end{document}